\documentclass[12pt]{article}
\usepackage{amsthm}
\usepackage{amsmath}
\usepackage{amsfonts}
\usepackage{graphicx}
\usepackage{latexsym}
\usepackage{amssymb}
\usepackage{enumerate}
\usepackage{comment}
\usepackage{array,booktabs,longtable}

\usepackage[hidelinks]{hyperref}
\usepackage[nameinlink,capitalise,noabbrev]{cleveref}
\usepackage{xcolor}
\newcommand{\mmod}{{\mbox{mod}}}

\newtheorem{theorem}{Theorem}
\newtheorem{lemma}{Lemma}
\newtheorem{corollary}{Corollary}

\newtheorem{proposition}{Proposition}

\numberwithin{theorem}{section}
\numberwithin{lemma}{section}
\numberwithin{corollary}{section}
\numberwithin{proposition}{section}

\begin{document}

\bibliographystyle{plain}

\title{The Last Seven Open Radii for Perfect Codes\\ in the Johnson Scheme}
\author{
{\sc Xiande Zhang}\thanks{School of Mathematical Sciences
University of Science and Technology of China, Hefei, Anhui, 230026, China,
e-mail:{\tt drzhangx@ustc.edu.cn}.} \and
{\sc Wenjie Zhong}\thanks{School of Mathematical Sciences
University of Science and Technology of China, Hefei, Anhui, 230026, China,
e-mail:{\tt 3083848454@qq.com}.}}
\date{}\maketitle

\begin{abstract}

 Delsarte (1973) conjectured that there are
no nontrivial perfect codes in the Johnson scheme.
In this paper, we prove that there are no nontrivial $e$-perfect codes in the Johnson scheme for $e\in\{1,2,4,9,10,12,16\}$. This paper complements \emph{Perfect Codes in the Johnson Scheme Hardly Exist}, thus proving Delsarte's conjecture completely.
\end{abstract}
\section{Introduction}

Perfect codes are extremal error-correcting codes whose metric balls form a partition of the ambient space. Their classification in the Hamming scheme is a classical topic in coding theory; see, for example, \cite{McSl77}. The corresponding problem for constant-weight codes is naturally formulated in the Johnson scheme.

For integers \(0\le w\le n\), the Johnson graph \(J(n,w)\) has vertex set \(\binom{[n]}{w}\) and two vertices are adjacent when their intersection has size \(w-1\). Equivalently, the Johnson distance between \(X,Y\in\binom{[n]}{w}\) is
\[
d(X,Y)=w-|X\cap Y|.
\]
A code \({C}\subseteq\binom{[n]}{w}\) is called \(e\)-perfect if the balls of Johnson radius \(e\) centred at the codewords of \({C}\) partition \(\binom{[n]}{w}\). Since complementation gives an isomorphism \(J(n,w)\cong J(n,n-w)\), we may assume throughout that \(n\ge 2w\) and write
\[
n=2w+\delta,\qquad \delta\ge 0.
\]

There are three standard families of trivial perfect codes: the whole vertex set, which is \(0\)-perfect; a singleton, which is \(w\)-perfect; and, when \(n=2w\) and \(w\) is odd, a complementary pair of \(w\)-subsets, which is \((w-1)/2\)-perfect.
Delsarte conjectured that these are the only perfect codes in the Johnson scheme \cite{Del73}.
This paper complements \cite{johnson-first}. Our main result excludes the seven radii $e=1,2,4,9,10,12,16$ left open there.

\begin{theorem}[Main theorem]\label{thm:main-result}
There is no nontrivial \(e\)-perfect code in any Johnson scheme \(J(n,w)\) for
\[
e\in\{1,2,4,9,10,12,16\}.
\]
\end{theorem}



The paper is organized as follows. \cref{subsec:section5-quoted-results} lists the known results as propositions, with their original source numbers. 
\cref{sec:prep} provides two useful lemmas. \cref{sec:complete-exclusion} excludes $e=2,9,10,12,16$,  \cref{sec:no-four-perfect-codes} excludes $e=4$ and \cref{sec:no-one-perfect-codes} excludes $e=1$.

\section{Known results and tools}\label{subsec:section5-quoted-results}
In this section, we collect all results that will be used in this paper. Most of which are from \cite{johnson-first}. The perfect-code statements are used with  $n\ge2w$.

The size of a ball of Johnson radius \(e\) is
\[
\Phi_e(n,w)=\sum_{i=0}^{e}\binom{w}{i}\binom{n-w}{i}.
\]
The sphere-packing identity therefore gives the elementary divisibility \(\Phi_e(n,w)\mid\binom{n}{w}\).
A different source of necessary conditions comes from the combinatorial structures contained in a perfect code. In particular, suitable families of codewords form embedded Steiner systems, whose parameter restrictions impose strong bounds on \(n\) and \(w\). See \cite{Etz96,Etz01,Etz06,Etz22,EtvL01,EtSc04}.

The strength $\varphi$ is the largest $t$ for which the code forms a $t$-design, and $d=w-\varphi$ is the strength redundancy.
Throughout, perfect codes are assumed to be nontrivial unless explicitly stated otherwise, and all logarithms are natural. 

\subsection{Previous results on perfect codes}
\begin{proposition}[Theorems~15 and~16 of \cite{EtSc04}]
\label{prop:lower bound of w}
Assume that there exists an $e$-perfect code in J$(n,w)$.
\begin{itemize}
\item If \(n>2w\) and \(n\) is odd, then \(w>\frac{e(e+1)(e+2)}{2}+2e+1\).
\item If \(n>2w\) and \(n\) is even, then \(w>e(e+1)(e+2)+2e+1\).
\item If \(n=2w\), then \(w>2e^2+4e+1\).
\end{itemize}
\end{proposition}

\begin{proposition}[Strict Roos bound: Theorem~13 of \cite{EtSc04}]\label{prop:strict-roos}
If a nontrivial \(e\)-perfect code exists in \(J(n,w)\), then
\[
n<(w-1)\frac{2e+1}{e}.
\]
\end{proposition}

\begin{proposition}[Propositions~3.2, 3.3 and~3.9, Eq.~(8) of \cite{johnson-first}]\label{prop:embedded-steiner}
If a nontrivial \(e\)-perfect code exists in \(J(n,w)\), then  for every \(0\le{j}\le e+1\),
\begin{equation}\label{eq:embedded-steiner-shift}
 \binom{e+{j}}{{j}}\mid\binom{w-e+{j}-1}{{j}},\qquad
 \binom{e+{j}}{{j}}\mid\binom{n-w-e+{j}-1}{{j}}.
\end{equation}
Further, \begin{equation}\label{eq-congruence-e+2}
2(w+2)\equiv 2(n-w+2)\equiv0\pmod{e+2}.
\end{equation}

\end{proposition}

\begin{proposition}[Lloyd factorization: Lemma~6.1  of \cite{johnson-first}]\label{lem:section5-Lloyd-factorization}Assume that there exists an $e$-perfect code in \textup{J}$(n,w)$.
There are distinct integers \(0\le h_1<\cdots<h_e=d-1\) such that
\begin{equation}\label{eq:section5-Lloyd-factorization}
(e!)^2\Phi_e(n,w)=\prod_{i=1}^e(w-h_i)(n-w+h_i+1).
\end{equation}
\end{proposition}

\begin{proposition}[Corollary~3.1  of \cite{johnson-first}]\label{cor:str_redund}
Suppose that a nontrivial \(e\)-perfect code in \(J(n,w)\) has strength \(\varphi\), and put \(d=w-\varphi\). If \(e\ge 2\), then
\[
d\le \lfloor e\sqrt{w}\rfloor.
\]
\end{proposition}


\begin{proposition}[Prime-power localization: Corollary~4.3 of \cite{johnson-first}]
\label{cor:prime factors from chain divisitilities}
If there exists an $e$-perfect code in $\textup{J}(n,w)$, then
\[
\Phi_e(n,w)\mid  \prod_{\substack{p\text{ prime,}\\ n-w<p^s\le n-w+d\text{ for some }s\ge 1}}p.
\]
In particular, \(\Phi_e(n,w)\) is squarefree.
 Moreover, for every prime \(p\mid\Phi_e(n,w)\), there exists an integer \(s\ge 1\), which is called a {\emph{localized exponent}} of $p$, such that
\[
n-w<p^s\le n-w+d.
\]
\end{proposition}

For each prime $p\mid\Phi_e(n,w)$, fix one localized exponent $s_p$ with $n-w<p^{s_p}\le n-w+d$, and  the \emph{reciprocal mass}
\begin{equation}\label{eq:reciprocal-mass-definition}
S=\sum_{p\mid\Phi_e(n,w)}\frac1{s_p}.
\end{equation}
Then the nonexistence of a $e$-perfect code is obtained either by finding a square in $\Phi_e(n,w)$ by \cref{cor:prime factors from chain divisitilities} or by incompatible bounds of $S$.

\begin{proposition}[Lemma~6.2  of \cite{johnson-first}]\label{lem:section5-Lloyd-mass}
If there exists an $e$-perfect code in $\textup{J}(n,w)$ with $e\geq 2$,
\begin{equation}\label{eq:section5-mass-bounds}
\frac{\log\Phi_e(n,w)}{\log M}\le S<\frac{\log\Phi_e(n,w)}{\log (n-w)}.
\end{equation}
Moreover, at most $e$ localized exponents are equal to $1$.
\end{proposition}

\begin{proposition}[Lemma~4.3 of \cite{johnson-first}]\label{lem:residue-localization}
Let $p$ be a prime, $P=p^k$, where $k\ge 1$, $b \equiv w~(\mmod~P)$, and $c \equiv n-w~(\mmod~P)$, such that $0 \leq b,c <P$, and suppose that
$$
e\equiv -\gamma \pmod P \qquad \text{and} \qquad 1\le \gamma\le \frac{P+1}{2}.
$$
If there exists an $e$-perfect code in \textup{J}$(n,w)$, then $P- 2\gamma+1\le b,c\le P-\gamma$.
\end{proposition}

\begin{proposition}[Theorem~4.2 of \cite{johnson-first}]\label{thm:modulo-p-criterion}
Let \(p\) be a prime. Suppose that, for some \(k\ge1\), \(P=p^k\), and \(1\le\gamma<P\),
$$
e\equiv-\gamma~(\mmod~P) \qquad \text{and}  \qquad 1 \leq \gamma\le\left\lfloor\frac{P+2}{4}\right\rfloor.
$$
If a $e$-perfect code exists in \textup{J}$(n,w)$, then $p$ divides $\Phi_e(n,w)$.
\end{proposition}

\begin{proposition}[Corollary~4.4 of \cite{johnson-first}]\label{cor:moving-prime-detector}
If \(e<p\le(4e+2)/3\) is prime and a nontrivial \(e\)-perfect code exists, then \(p\mid\Phi_e(n,w)\).
\end{proposition}


\subsection{Some useful tools in number theory}
%
Two old theorems in number theory will be used, Kummer's theorem~\cite[p. 245]{GKP94,Gra97} and Lucas' theorem~\cite{Fin47,Luc91}.

\begin{theorem} [Lucas' theorem]
\label{thm:Lucas}
Let $p$ be a prime and let
$$
N = \prod_{j \geq 0} N_j p^j,  ~~~ K = \prod_{j \geq 0} K_j p^j, ~~~ 0 \leq N_j , K_j <p ,
$$
then
$$
\binom{N}{K} \equiv \prod_{j \geq 0} \binom{N_j}{K_j} ~ (\mmod ~ p).
$$
\end{theorem}
\begin{corollary}
\label{cor:Lucas}
If $p$ is a prime and
$$
N = \prod_{j \geq 0} N_j p^j,  ~~~ K = \prod_{j \geq 0} K_j p^j, ~~~ 0 \leq N_j , K_j <p ,
$$
then $p$ does not divide $\binom{N}{K}$ if and only if $K_j \leq N_j$ for every $j$.
\end{corollary}

\begin{theorem} [Kummer's theorem]
\label{thm:Kummer}
Let $p$ be a prime. The number of times $p$ appears in the factorization of $\binom{a}{b}$ equals the number of carries when adding
$b$ to $a-b$ in base $p$.
\end{theorem}

We next record the effective two-logarithm estimate used to cap the localized exponents.
For a prime $p$, put $\ell(p)=\max\{1,\log p\}$.

\begin{lemma}[Laurent's two-prime estimate,  Lemma~6.5 of \cite{johnson-first}]
\label{lem:section5-Laurent}
Let $p\ne q$ be primes and $u,v\ge1$ be integers. Let
\[\Lambda=u\log p-v\log q,\qquad B=\frac{u}{\ell(q)}+\frac{v}{\ell(p)},\qquad H=\max\{\log B+0.38,10\}.\]
Then $\Lambda\ne0$ and
\begin{equation}\label{eq:section5-Laurent}
\log|\Lambda|\ge-25.2H^2\ell(p)\ell(q).
\end{equation}
\end{lemma}


%

\begin{proposition}[Theorem~5 of \cite{rosser1962approximate}]\label{prop:quoted-RS-reciprocal}
Let $B\approx0.261497212847643$.
Then for any integer $x\geq 286$,
\begin{equation}\label{eq:RS-317}
\sum_{\substack{p\le x\\p\text{ prime}}}\frac1p
<\log\log x+B+\frac{1}{2\log^2x}.
\end{equation}
\end{proposition}

Consequently, \cref{prop:quoted-RS-reciprocal} implies for $x\geq 286$
\begin{equation}\label{eq:section5-Dusart-simple}
\sum_{\substack{p\le x\\p\text{ prime}}}\frac1p<\log\log x+0.278\qquad(x\ge286),
\end{equation}
since the term $B+\frac{1}{2\log^2x}$ decreases with $x$ and total less than $0.278$ at $x=286$.


Rosser and Schoenfeld \cite{rosser1962approximate} defined $\psi(x)$ as the logarithm of the least common multiple of all positive integers not exceeding $x$ (and put $\psi(x)=0$ for $x<2$). For an integer $N$, this definition gives
$\psi(N)=\log\operatorname{lcm}(1,\ldots,N)$. We quote exactly the estimate needed below.

\begin{proposition}[Theorem~12 of \cite{rosser1962approximate}]\label{prop:s7-RS-psi}
The quotient $\psi(x)/x$ takes its maximum at $x=113$, and
\begin{equation}\tag{3.35}\label{eq:RS-335}
\psi(x)<1.03883x\qquad(0<x).
\end{equation}
\end{proposition}

\begin{proposition}[Legendre's criterion, Theorem~184 of \cite{HardyWright1979}]\label{thm:quoted-legendre} Let $x$ be irrational and let $\frac ab$ be a reduced rational number with $b\geq 1$.
If
\[
\left|\frac ab-x\right|<\frac1{2b^2},
\]
then \(a/b\) is a convergent of the simple continued fraction of \(x\).
\end{proposition}

\section{Preparations}\label{sec:prep}
In this section, we prepare some preliminary results that will be used frequently later.
First, we collect our notations
\[
n=2w+\delta,\qquad \delta\ge0,\qquad {\bar\omega}=n-w=w+\delta,\qquad M={\bar\omega}+d.
\]
We shall repeatedly use two bounds which are valid for every \(n\ge2w\) and $e\geq 1$.  First, the strict Roos bound in \cref{prop:strict-roos} gives
\begin{equation}\label{eq:section5-w-over-z}
\frac w{\bar\omega}>\frac e{e+1}.
\end{equation}
Second, \cref{prop:lower bound of w} implies
\begin{equation}\label{eq:section5-universal-weight}
\begin{aligned}
  w>&V_e:=2e^2+4e+1, \qquad n\geq 2w.
\end{aligned}
\end{equation}
Thus \({\bar\omega}\ge w>2e^2\), and ${\bar\omega}<M={\bar\omega}+d<2{\bar\omega}.$
By  \cref{cor:prime factors from chain divisitilities},
\begin{equation}\label{eq:section5-input}
 {\bar\omega}<p^{s_p}\le M\quad\text{for every prime }p\mid\Phi_e(n,w),
\end{equation}
where \(\Phi_e(n,w)\) is squarefree and one localized exponent \(s_p\ge1\) is fixed for each prime divisor \(p\). Finally, \cref{cor:moving-prime-detector} shows that every prime in
\begin{equation}\label{eq:section5-moving-primes}
\mathcal P_e:=\left\{p\text{ prime}:e<p\le Q_e:=\frac{4e+2}{3}\right\}
\end{equation}
divides \(\Phi_e(n,w)\).

When $e\geq 2$, by \cref{cor:str_redund},
\begin{equation}\label{eq:section5-input1}
d\le e\sqrt w.
\end{equation}
 Combining \({\bar\omega}\ge w>2e^2\),  we have when $e\geq 2$,
\begin{equation}\label{eq:section5-M-less-2z}
\frac d{\bar\omega}\le\frac e{\sqrt {\bar\omega}}<\frac1{\sqrt2}.
\end{equation}

\begin{lemma}[General exponent caps]\label{lem:section5-general-cap}
Suppose that an $e$-perfect code exists in $\textup{J}(n,w)$ with $e\geq 2$,  
and there exists an odd prime divisor $q$ of $\Phi_e(n,w)$ such that
 \[5\log e+3\eta_e\leq 100 \log q,\]
  where $\eta_e:=\log(1+e/\sqrt {\bar\omega})$. Then for any prime divisor $p\neq q$ of  $\Phi_e(n,w)$,
its localized exponent
\[
s_p<
\begin{cases}
5040\log q+\frac{2\log e+\eta_e}{\log p}, & p \text{ odd},\\
5040\frac{\log q}{\log 2}+\frac{2\log e+\eta_e}{\log 2}, & p=2.
\end{cases}
\]


\end{lemma}

\begin{proof} For any prime divisor $p\neq q$, denote $u=s_p$. 
Write $q^v$ for the  localized power of $q$. Since the two powers $p^u,q^v$ are distinct and lie in $(\bar\omega, \bar\omega+d]$,
\begin{equation}\label{eq:section5-Lambda-upper}
0<|u\log p-v\log q|\le \log\left(1+\frac {d-1}{\bar\omega+1}\right)<\frac {d-1}{\bar\omega+1}\le\frac e{\sqrt {\bar\omega}}.
\end{equation}
Combining this with \cref{lem:section5-Laurent} gives
\begin{equation}\label{eq:section5-logz-Laurent}
\log {\bar\omega}<2\log e+50.4H^2\ell(p)\ell(q).
\end{equation}
Denote $\eta_e:=\log(1+e/\sqrt {\bar\omega})$.
The localization ${\bar\omega}<p^u,q^v\le {\bar\omega}+d \leq {\bar\omega}(1+e/\sqrt {\bar\omega})$ gives
\begin{equation}\label{eq:uv}
 u\le\frac{\log {\bar\omega}+\eta_e}{\log p},\qquad v\le\frac{\log {\bar\omega}+\eta_e}{\log q}.
\end{equation}

If both $p,q$ are odd, the parameter $B$ in \cref{lem:section5-Laurent}  satisfies
\[
B=\frac{u}{\log q}+\frac{v}{\log p}< \frac{2(\log {\bar\omega}+\eta_e)}{\log p\log q}< 100.8H^2+ \frac{4\log e+2\eta_e}{\log p\log q}< 100.8H^2+ \frac{4\log e+2\eta_e}{\log q}.
\]
If $p=2$ (i.e., $\ell(p)=1$) and $q$ is odd, then
\[
B=\frac{u}{\log q}+v < \frac{(\log {\bar\omega}+\eta_e)}{\log q}(1+\frac{1}{\log 2})< 123.5H^2+ 2.45\frac{2\log e+\eta_e}{\log q}<123.5H^2+ \frac{5\log e+3\eta_e}{\log q}.
\]
Since $5\log e+3\eta_e\leq 100 \log q$, if $H\geq 10$, for both cases $B\leq 130H^2$.
If $H>10$, then $B=\exp(H-0.38)$, whereas $\exp(x-0.38)/x^2>150$ at $x=10$ and increases for $x>2$. Hence $B>150H^2$, contradicting $B\leq 130H^2$. Therefore, $H=10$. Hence by \eqref{eq:uv}, we obtain the required upper bound of $u$.
\end{proof}

The next elementary packing lemma is used  for upper-bounding the reciprocal-mass.

\begin{lemma}[Exponent packing]\label{lem:section5-exponent-packing}
Let $\mathcal E$ be a finite set of integers at least $3$ and at most $N$ such that $\gcd(s,t)\le2$ whenever $s\ne t\in \mathcal E$. 
Then
\begin{equation}\label{eq:section5-pack-gcd2}
\sum_{s\in\mathcal E}\frac1s\le\frac14+\sum_{\substack{3\le r\le N\\r\text{ prime}}}\frac1r.
\end{equation}
\end{lemma}

\begin{proof}
For $s\in\mathcal E$, put $m_s=s$ when $s$ is odd and $m_s=s/2$ when $s$ is even. Then $m_s\ge 2$ as $s\ge 3$. Since $\gcd(s,t)\le2$, $\gcd(m_s,m_t)\le 2$. However, $\gcd(m_s,m_t)= 2$ would make $4\mid\gcd(s,t)$, which is impossible. So the integers $m_s$ are pairwise coprime,
which implies that at most one $m_s$ is even. Equivalently,  at most one $s$ is divisible by $4$, which contributes at most $1/4$. For every remaining $m_s$, choose an odd prime divisor $r_s\mid m_s$. Pairwise coprimality makes these primes distinct, and $r_s\le m_s\le s\le N$. Since $1/s\le1/r_s$, summing gives \eqref{eq:section5-pack-gcd2}.
\end{proof}

\section{Exclusion of radii 2, 9, 10, 12 and 16}\label{sec:complete-exclusion}

\begin{lemma}\label{lem:section5-comparison-primes}
For $e=9,10,12,16$, respectively, the following primes divide $\Phi_e(n,w)$:
\[
2,5,11\quad(e=9),\qquad 11,13\quad(e=10),\qquad 2,7,13\quad(e=12),\qquad 3,17,19\quad(e=16).
\]
\end{lemma}

\begin{proof}
For the four small radii, apply \cref{thm:modulo-p-criterion}.  The parameters $(p,P,\gamma)$ are
\[
\begin{array}{c|c}
e&(p,P,\gamma)\\ \hline
9&(2,2,1),(5,5,1),(11,11,2)\\
10&(11,11,1),(13,13,3)\\
12&(2,16,4),(7,7,2),(13,13,1)\\
16&(3,9,2),(17,17,1),(19,19,3).
\end{array}
\]

\end{proof}

The next proposition provides the uniform weight bound \(w>e^4\). Combined with \(d\le e\sqrt w\), this makes the localization interval \(({\bar\omega},{\bar\omega}+d]\) sufficiently short for the reciprocal-mass and exponent-spacing arguments used in the remainder of this section.

\begin{proposition}\label{prop:section5-weight-amplification}
If
$e\in\{9,10,12,16\}$
and a nontrivial $e$-perfect code in $J(n,w)$ exists, then
\begin{equation}\label{eq:section5-weight-bounds}
w>e^4.
\end{equation}
\end{proposition}

\begin{proof}
Throughout, suppose for contradiction that $w\le e^4$.  By \eqref{eq:section5-w-over-z} and $d\le e\sqrt w$,
\begin{equation}\label{eq:section5-quartic-basic-upper}
{\bar\omega}<e^4+e^3,\qquad M={\bar\omega}+d<e^4+2e^3.
\end{equation}
For $e=9$, \cref{lem:section5-comparison-primes} gives prime divisors $5,11$, \eqref{eq:section5-universal-weight} gives ${\bar\omega}>V_9=199>11^2,5^3$, while \eqref{eq:section5-quartic-basic-upper} gives $M<8019<11^4$.  Hence the localized $11$-power is $11^3=1331$, so ${\bar\omega}<1331$.  Therefore $M={\bar\omega}+d<1331+9\sqrt{1331}<1660<5^5$.  The localized $5$-power is consequently $5^4=625$, so ${\bar\omega}<625$. The localization that $11^3,5^4$ lies in the interval $({\bar\omega},{\bar\omega}+d]$ gives
\[
706=11^3-5^4<d\le9\sqrt {\bar\omega}<9\sqrt{625}=225,
\]
a contradiction.

For $e=10$, \cref{lem:section5-comparison-primes} gives prime divisors $11,13$, \eqref{eq:section5-universal-weight} gives ${\bar\omega}>V_{10}=241>13^2$, while \eqref{eq:section5-quartic-basic-upper} gives $M<12000<11^4$.  Thus the localized $11$- and $13$-powers are $11^3$ and $13^3$.  In particular ${\bar\omega}<11^3=1331$, and hence
\[
866=13^3-11^3<d\le10\sqrt {\bar\omega}<10\sqrt{1331}<365,
\]
a contradiction.

For $e=12$, \cref{lem:section5-comparison-primes} gives prime divisors $2,7,13$,  \eqref{eq:section5-universal-weight} gives ${\bar\omega}>V_{12}=337>13^2$, while \eqref{eq:section5-quartic-basic-upper} gives $M<24192<13^4$.  Hence the localized $13$-power is $13^3=2197$, so ${\bar\omega}<2197$ and
\[
M={\bar\omega}+d<{\bar\omega}+12\sqrt {\bar\omega}<2197+12\sqrt{2197}<2760.
\]
Since $M<2760<7^5,2^{12}$ and ${\bar\omega}=M-d> 13^3-12\sqrt{2197}>1634>7^3,2^{10}$, the localized $7$- and $2$-power are $7^4=2401$ and $2^{11}=2048$. Consequently
\[
{\bar\omega}<2048,\qquad {\bar\omega}=M-d>7^4-12\sqrt{2048}>1857.
\]
We now use the residue localization already contained in \cref{lem:residue-localization}.  For the three detecting parameters $(p,P,\gamma)=(13,13,1),(7,7,2),(2,16,4)$ given by \cref{thm:modulo-p-criterion},
\[
w,{\bar\omega}\equiv12\pmod{13},\qquad w,{\bar\omega}\equiv4\text{ or }5\pmod7,\qquad w,{\bar\omega}\equiv9,10,11,\text{ or }12\pmod{16}.
\]
\cref{prop:embedded-steiner} with $e+2=14$ sharpens the congruence modulo $7$ to $w\equiv{\bar\omega}\equiv5\pmod7$. 
A short calculation gives the possible residues modulo \(13\cdot7\cdot16=1456\). Writing \(w\text{ or }{\bar\omega}=12+13t\), the congruence modulo \(7\) gives \(t\equiv0\pmod7\), while the congruence modulo \(16\) gives \(t\equiv0,1,6,11\pmod{16}\).  Combining these modulo \(7\cdot16=112\) yields
\[
t\equiv0,49,70,91\pmod{112},
\]
and hence
\[
w,{\bar\omega}\pmod{1456}\in\{12,649,922,1195\}.
\]
None of these residue classes has a representative in $1857<{\bar\omega}<2048$, a contradiction.

For $e=16$, \cref{lem:section5-comparison-primes} gives prime divisors $17,19$, \eqref{eq:section5-universal-weight} gives ${\bar\omega}>V_{16}=577>19^2$, while \eqref{eq:section5-quartic-basic-upper} gives $M<73728<17^4$.  Thus the detected $17$- and $19$-powers are both cubes.  Hence ${\bar\omega}<17^3=4913$, but
\[
19^3-17^3=1946<d\le16\sqrt {\bar\omega}<16\sqrt{4913}<1122,
\]
a contradiction.

Therefore, the quartic bound follows in every case.
\end{proof}

\begin{lemma}[Odd exponent caps]\label{lem:section5-moving-cap}
If $e\in\{9,10,12,16\}$ and $p$ is an odd prime divisor of $\Phi_e(n,w)$, then
\begin{equation}\label{eq:section5-small-odd-cap}
s_p<15000.
\end{equation}
\end{lemma}

\begin{proof}
Fix an odd prime $p\mid\Phi_e$ and put $u=s_p$. We can always choose a detected odd prime $q\ne p$ from \cref{lem:section5-comparison-primes}. Since ${\bar\omega}\ge w>e^4$ by \cref{prop:section5-weight-amplification},   $\eta_e=\log(1+e/\sqrt {\bar\omega})<\log(1+1/e)$. Then
\[\frac{5\log e+3\eta_e}{\log q}\leq \frac{5\log 16+3 \log(1+\frac{1}{9})}{\log 3}<13.\]
By \cref{lem:section5-general-cap}, we have
\begin{equation}\label{eq:section5-u-small-cap}
u<5040\log q+\frac{2\log e+\eta_e}{\log p}<5040\log q+\frac{2\log e+\log(1+1/e)}{\log p}.
\end{equation}


For $e=9$, use $q=5$ unless $p=5$, and otherwise use $q=11$. Then \eqref{eq:section5-u-small-cap} with the maximum case $(p,q)=(5,11)$ is less than $12089$. For $e=10$, use $q=11$ unless $p=11$, in which case use $q=13$; then \eqref{eq:section5-u-small-cap} with the maximum case $(p,q)=(11,13)$ is then less than $12930$. For $e=12$, use $q=7$ unless $p=7$, and otherwise use $q=13$; the maximum $u$ with $(p,q)=(7,13)$ is less than $12931$. For $e=16$, use $q=3$ unless $p=3$, and otherwise use $q=17$; the maximum $u$ with $(p,q)=(3,17)$ is less than $14285$. Thus $s_p<15000$ for all four radii.

\end{proof}

\subsection{Exclusion of the radii}
We now start to exclude the radii less than $32$ that survive the prime-square sieve in the first paper. The radius \(e=2\) is treated separately by a continued-fraction argument. For \(e=9,10,12,16\), we use the reciprocal-mass method: the short localization interval forces \(\gcd(s_p,s_q)\le2\) for distinct odd prime divisors, so the exponent-packing lemma, the uniform exponent bound and Dusart's estimate give a numerical upper bound for \(S-e\), which is incomparable with the corresponding lower bound.

\begin{proposition}[Etzion--Schwartz, Proposition~3 of \cite{EtSc04}]\label{prop:quoted-etzion-radius2}
There are no \(2\)-perfect codes in \(J(n,w)\) for all \(n\le40000\).
\end{proposition}

\begin{theorem}[The radius $2$]\label{thm:section5-two}
There is no nontrivial $2$-perfect code in any Johnson graph.
\end{theorem}

\begin{proof}
For $e=2$, \eqref{eq-congruence-e+2} implies that $w$ and ${\bar\omega}$ are both even, and \eqref{eq:embedded-steiner-shift} with $j=2$ implies that  $4\mid(w-1)(w-2)$ and $4\mid({\bar\omega}-1)({\bar\omega}-2)$. Hence, $w\equiv {\bar\omega}\equiv2\pmod4$.
So \[
\Phi_2=1+w{\bar\omega}+\binom w2\binom {\bar\omega}2\equiv1+0+1\equiv0\pmod2.
\]
\cref{thm:modulo-p-criterion} with $(p,P,\gamma)=(3,3,1)$ also gives $3\mid\Phi_2$. Hence $2,3\mid\Phi_2$.  Choose localized powers
\[
{\bar\omega}<2^u,3^v\le {\bar\omega}+d.
\]
By  \cref{prop:quoted-etzion-radius2}, no $2$-perfect code exists for $n\le40000$; hence we assume ${\bar\omega}\ge n/2>20000$, so $u\ge15$.  

First we show an upper bound of the exponent $u$ of $2$. In \cref{lem:section5-general-cap}, $\eta_2=\log(1+2/\sqrt {\bar\omega})<\log(1+2/\sqrt{20000})< 0.015$. Then
\[\frac{5\log e+3\eta_e}{\log 3}\leq \frac{5\log 2+3 \times 0.015}{\log 3}<4.\]
By \cref{lem:section5-general-cap}, we have
\begin{equation}\label{eq:section5-u-basic}
u<5040\frac{\log 3}{\log 2}+\frac{2\log 2+0.015}{\log 2}<7991.
\end{equation}


We next apply the continued fraction argument to give possible values of $v/u$.  Since $2^u\le M<2{\bar\omega}$, we have ${\bar\omega}>2^{u-1}$ and hence by $|\Lambda|<2/\sqrt {\bar\omega}$ from \eqref{eq:section5-Lambda-upper},
\begin{equation}\label{eq:section5-Legendre-two}
\left|\frac{\log2}{\log3}-\frac vu\right|= \frac{|\Lambda|}{u\log3}<\frac{2}{u\log3\cdot\sqrt {\bar\omega}} <\frac{2\sqrt2}{u\log3\cdot2^{u/2}}.
\end{equation}
Since $u\ge 15$ and the function $2^{u/2}/u$ is increasing for $u\ge3$, we have $4\sqrt2u<\log 3\cdot 2^{u/2}$. Hence
\begin{equation}\label{eq:section5-Legendre-two}
\left|\frac{\log2}{\log3}-\frac vu\right|<\frac{1}{2u^2}.
\end{equation}
Applying Legendre's criterion  in \cref{thm:quoted-legendre} with $x=\frac{\log2}{\log3}$ and $a/b$ being the reduced fraction associated with $v/u$,
we get that  $a/b$ is a convergent of $\log2/\log 3$.
The required continued-fraction prefix is
\[
\frac{\log2}{\log3}=[0;1,1,1,2,2,3,1,5,2,23,\ldots].
\]
The positive convergents whose denominators lie below $7991$ are in
\[
D=\{1/1,\ 1/2,\ 2/3,\ 5/8,\ 12/19,\ 41/65,\ 53/84,\ 306/485,\ 665/1054\},
\]
and the next denominator is $24727$. That is, $v/u$ can only equal to someone in $D$.

Finally, we show that $v/u$ can not equal any one in $D$. Observe that
 \[R_{u,v}:=\frac{|2^u-3^v|^2}{4\min\{2^u,3^v\}}\leq \frac{d^2}{4\bar{\omega}}\leq \frac{4w}{4\bar{\omega}}\leq 1.\]
If $v/u=a/b$ for some displayed convergent $a/b$ in $D$, then $(v,u)=(ag,bg)$ with $g\geq 1$. If $2^b<3^a$, define
\[
\rho_{u,v}(y):=\frac{(3^{ay}-2^{by})^2}{4\cdot 2^{by}}=\frac{2^{by}}{4}\left(\left(\frac {3^a} {2^b}\right)^y-1\right)^2;
\]
 if $2^b>3^a$, just interchange $2^b$ and $3^a$. Note that $\rho_{u,v}(g)=R_{u,v} \leq 1$.
Since $u\geq 15$, and $\rho_{u,v}(y)$ is strictly increasing in $y\ge1$, it suffices to check for any $a/b\in D$, the values of $\rho_{u,v}(y)$ for the least integer $y$ such that $b y\ge15$. These corresponding pairs $(a y,b y)$ for each $a/b\in D$ are
\[
(15,15),(8, 16),(10 ,15),(10 ,16),(12 ,19),(41 ,65),(53 ,84),(306 ,485),(665 ,1054).
\]
Direct substitution gives $\rho_{u,v}>1$ in every case (in fact much bigger than 1).  This contradiction excludes $e=2$.
\end{proof}

The same reciprocal-mass proof now handles all four remaining small radii.  Once the quartic weight bound is available, the basic estimate $d\le e\sqrt {\bar\omega}$ is already sufficient for the proof.

\begin{proposition}[Small radii by reciprocal mass]\label{prop:section5-small-unified}
There is no nontrivial $e$-perfect code in any Johnson graph for
\[
e\in\{9,10,12,16\}.
\]
\end{proposition}

\begin{proof}
By \cref{prop:section5-weight-amplification,eq:section5-w-over-z,eq:section5-input1},
\begin{equation}\label{eq:section5-tail-weight}
{\bar\omega}\ge w>e^4,\qquad \frac w{\bar\omega}>\frac e{e+1}, \qquad d\le e\sqrt {\bar\omega}.
\end{equation}
Put $c_e=\frac e{e+1}-\frac 1e$.
Then $d/{\bar\omega}<1/e$, $M={\bar\omega}+d<{\bar\omega}(1+1/e)$, and $w-h_i>w-d>c_e{\bar\omega}$ with $h_i$ from \cref{lem:section5-Lloyd-factorization}. Lloyd factorization of \cref{lem:section5-Lloyd-factorization} gives $\Phi_e>(c_e{\bar\omega}^2)^e/(e!)^2$, and thus
\[
\log\Phi_e>2e\log {\bar\omega}+e\log c_e-2\log(e!).
\]

Using \cref{lem:section5-Lloyd-mass} with $\log M<\log {\bar\omega}+\log(1+1/e)$ and subtracting the possible contribution $e$ of exponent $1$,
\begin{equation}\label{eq:section5-small-lower-unified}
S-e\ge \frac{\log\Phi_e}{\log M}-e> \frac{e\log {\bar\omega}+e\log c_e-2\log(e!)-e\log(1+1/e)}{\log {\bar\omega}+\log(1+1/e)}.
\end{equation}
The right side increases with $\log {\bar\omega}$, because its derivative has numerator $2e\log(1+1/e)-e\log c_e+2\log(e!)>0$.  We may therefore put ${\bar\omega}=e^4$ on its right-hand side.

To obtain the upper side of $S-e$, we first bound the remaining exponent classes.  If $p,q$ are distinct odd prime divisors of $\Phi_e$ and $g=\gcd(s_p,s_q)\ge3$, write $p^{s_p}=A^g$ and $q^{s_q}=B^g$.  Since $A,B$ are distinct odd integers, $|A-B|\ge2$, and the mean-value theorem gives
\[
|A^g-B^g|>2g {\bar\omega}^{1-1/g}\ge6{\bar\omega}^{2/3}.
\]
For each $e\in\{9,10,12,16\}$, one has $6e^{2/3}>e$, so ${\bar\omega}>e^4$ implies $6{\bar\omega}^{2/3}>6e^{2/3}\sqrt {\bar\omega}>e\sqrt {\bar\omega}\ge d$, contradicting the localization. Thus
\begin{equation}\label{eq:section5-small-gcd2}
\gcd(s_p,s_q)\le2
\end{equation}
for distinct odd prime divisors.

Let $N_2$ be the number of odd prime divisors with exponent $2$. If their bases are $p_1<\cdots<p_{N_2}$, then the mean-value theorem gives
\[
p_{j+1}^2-p_j^2>4p_j>4\sqrt {\bar\omega}.
\]
Summing over all $p_{j+1}^2-p_j^2$ in the localization interval gives $4(N_2-1)\sqrt {\bar\omega}<d\le e\sqrt {\bar\omega}$, so the exponent-$2$ contribution $N_2/2$ is less than $1/2+e/8$.

For odd-prime exponents $s_p\ge3$, \cref{lem:section5-moving-cap} gives $s_p<15000$.  By \eqref{eq:section5-small-gcd2}, \cref{lem:section5-exponent-packing,eq:section5-Dusart-simple} apply.  The primes chosen in the proof of \cref{lem:section5-exponent-packing} are odd, so the term $1/2$ belonging to the prime $2$ may be removed from \eqref{eq:section5-Dusart-simple}. Consequently
\[
\sum_{\substack{p\mid\Phi_e\\p\text{ odd},\ s_p\ge3}}\frac1{s_p}<\frac14+\left(\log\log15000+0.278-\frac12\right)=\log\log15000+0.028<2.292.
\]
If $2\mid \Phi_e$, its contribution by $2^{s_2}>{\bar\omega}$ is $1/s_2<\log2/\log {\bar\omega}<\log2/(4\log e)$. Hence
\begin{equation}\label{eq:section5-small-upper-unified}
S-e<\frac12+\frac e8+2.292+\frac{\log2}{4\log e}.
\end{equation}
Evaluating \eqref{eq:section5-small-lower-unified} at ${\bar\omega}=e^4$ gives
\[
\begin{array}{c|c|c}
e&\text{lower bound for }S-e&\text{upper bound from \eqref{eq:section5-small-upper-unified}}\\ \hline
9&>5.66&<4.00\\
10&>6.32&<4.12\\
12&>7.60&<4.37\\
16&>10.13&<4.86.
\end{array}
\]
Every row is contradictory.
\end{proof}

\section{No 4-perfect codes}\label{sec:no-four-perfect-codes}

We now exclude radius $4$, using the general reciprocal-mass and exponent-packing lemmas of Sections~2 and~3. Suppose that a nontrivial $4$-perfect code exists in $J(2w+\delta,w)$, where $\delta\ge0$, and retain the notation
\[
{\bar\omega}=w+\delta,\qquad M={\bar\omega}+d,\qquad \Phi_4=\Phi_4(2w+\delta,w),\qquad S=\sum_{p\mid\Phi_4}\frac1{s_p}.
\]
By \cref{prop:lower bound of w,prop:strict-roos,cor:str_redund,cor:prime factors from chain divisitilities},
\begin{equation}\label{eq:section6-input}
w\ge50,\qquad {\bar\omega}<\frac54w,\qquad d\le4\sqrt w\le4\sqrt {\bar\omega},\qquad {\bar\omega}<p^{s_p}\le M\quad(p\mid\Phi_4),
\end{equation}
and $\Phi_4$ is squarefree. Moreover, \cref{thm:modulo-p-criterion} with $(p,P,\gamma)=(5,5,1)$ gives $5\mid\Phi_4$. In particular, a power of $5$ is localized in $({\bar\omega},M]$. The mass bounds in \cref{lem:section5-Lloyd-mass} apply, and at most four localized exponents equal $1$.

\begin{lemma}[Preliminary weight bound]\label{lem:section6-weight}
A hypothetical nontrivial $4$-perfect code satisfies $w>10^4$.
\end{lemma}

\begin{proof}
First we have $w\ge 50$ from \eqref{eq:section6-input}. The last summand of $\Phi_4$ and \eqref{eq:section6-input} give
\[
\Phi_4\ge\binom w4\binom {\bar\omega}4>\frac{(w-2)^8}{576},\qquad M={\bar\omega}+d<\frac54w+4\sqrt w.
\]
For any number $W$ with $50\le W\le w$, taking logarithm gives
\[
\log\Phi_4>8\log w+8\log(1-2/W)-\log576,\qquad \log M< \log w+\log(5/4+4/\sqrt W).
\]
Consequently, \cref{lem:section5-Lloyd-mass} implies
\[
S\ge \frac{\log\Phi_4}{\log M}> \frac{8\log w+8\log(1-2/W)-\log576}{\log w+\log(5/4+4/\sqrt W)}.
\]
The right side of the inequality increases with $\log w$. Hence
\begin{equation}\label{eq:section6-R4}
S>R_4(W):=\frac{8\log W+8\log(1-2/W)-\log576}{\log W+\log(5/4+4/\sqrt W)}.
\end{equation}
Now consider the localized exponents of primes dividing $\Phi_4$.
For each fixed $s\ge2$, at most one odd prime can have exponent $s$: if $p<q$ are odd primes whose $s$th powers are both localized, then
\[
q^s-p^s>2s {\bar\omega}^{1-1/s}\ge4\sqrt {\bar\omega}\ge d,
\]
contradicting the localization. We use this observation in the following two ranges.

Suppose first that $50\le w\le256$. Then $M<\frac54w+4\sqrt w\le 384<5^4$, whereas ${\bar\omega}\ge50>5^2$, so the localized power of $5$ must be $5^3$. Thus the localization gives ${\bar\omega}<5^3=125$, and thus $M={\bar\omega}+d<125+4\sqrt{125}<170<7^3$. Every prime $p\ge7$ therefore has $s_p\le2$. Also $s_2\ge6$ and $s_3\ge4$ from $2^{s_2},3^{s_3}>{\bar\omega}\ge 50$ whenever these primes divide $\Phi_4$. Hence
\[
S\le4+\frac12+\frac16+\frac14+\frac13=5.25<5.4593<R_4(50),
\]
contradicting \eqref{eq:section6-R4}.

Next suppose that $256<w\le10^4$. Now $M<\frac54w+4\sqrt w\le12900<7^5$, so every prime $p\ge7$ has $s_p\le4$. Since ${\bar\omega}\ge w>256$, the primes $2,3,5$, when present, satisfy $s_2\ge9$, $s_3\ge6$, and $s_5\ge4$. Therefore
\[
S\le4+\frac12+\frac13+\frac14+\frac19+\frac16+\frac14=\frac{101}{18}<5.612<6.3762<R_4(256),
\]
again a contradiction. This proves $w>10^4$.
\end{proof}

\begin{lemma}[Radius-four exponent cap]\label{lem:section6-exponent-cap}
For every prime $p\ne2,5$ dividing $\Phi_4$, one has $s_p<8115$.
\end{lemma}

\begin{proof}
Put $u=s_p$, and write $5^v$ for the localized power of $5$.
Since by \cref{lem:section6-weight}, ${\bar\omega}>10^4$,   $\eta_4=\log(1+4/\sqrt {\bar\omega})<\log(26/25)$. Then
\[\frac{5\log e+3\eta_e}{\log 5}\leq \frac{5\log 4+3 \log(26/25)}{\log 5}<5.\]
By \cref{lem:section5-general-cap}, we have
\begin{equation}\label{eq:section5-u-basic}
u<5040\log q+\frac{2\log e+\eta_e}{\log p}<5040\log 5+\frac{2\log 4+\log(26/25)}{\log 3}<8115.
\end{equation}

%
\end{proof}

\begin{theorem}[Nonexistence of radius-four perfect codes]\label{thm:no-four-perfect-codes}
There are no nontrivial $4$-perfect codes in any Johnson graph.
\end{theorem}

\begin{proof}
By \cref{lem:section6-weight} and \eqref{eq:section6-R4},
\begin{equation}\label{eq:section6-mass-lower}
S>R_4(10^4)>7.1130.
\end{equation}
We derive an incompatible upper bound. If distinct odd primes $p,q\ne5$ have $g=\gcd(s_p,s_q)\ge2$, write $p^{s_p}=A^g$ and $q^{s_q}=B^g$. Since $A,B$ are distinct integers with $|A-B|\ge2$ and $\min\{A,B\}>{\bar\omega}^{1/g}$,
\[
|A^g-B^g|> g{\bar\omega}^{1-1/g}|A-B|\ge 4\sqrt {\bar\omega}\ge d.
\]
This contradicts localization. Thus $\gcd(s_p,s_q)=1$, and in particular no exponent $s\ge2$ repeats among the prime divisors other than $2,5$.
For each $p\ne2,5$ with exponent $s_p\ge 2$, take a prime $r_p$ dividing $s_p$, then $1/s_p\le 1/r_p$ and all such primes $r_p$ are distinct as $\gcd(s_p,s_q)\le1$. Also since \cref{lem:section6-exponent-cap} gives $s_p<8115$ for $p\ne2,5$, applying \eqref{eq:section5-Dusart-simple} yields
\[
\sum_{\substack{p\mid\Phi_4\\p\ne2,5,\ s_p\ge2}}\frac1{s_p}\le\sum_{\substack{q<8115\\q\text{ prime}}}\frac1q<\log\log8115+0.278.
\]
Exponent $1$ contributes at most $4$ by \cref{lem:section5-Lloyd-mass}. $2^{s_2},5^{s_5}>{\bar\omega}>10^4$ gives $1/s_2<\log2/\log(10^4),\allowbreak 1/s_5<\log5/\log(10^4)$.
Summing the contributions,
\[
S<4+\log\log8115+0.278+\log2/\log(10^4)+\frac{\log5}{\log(10^4)}<6.7254,
\]
contradicting \eqref{eq:section6-mass-lower}.
\end{proof}


\section{No \texorpdfstring{$1$}{1}-perfect codes}\label{sec:no-one-perfect-codes}

We now treat radius $1$. The proof has three stages. First, we consider separately the reciprocal sums associated with the two factors of the sphere size, and use their denominators to prove that the larger factor is prime. Next, the remaining bounded localized exponents exclude the large-weight range. Finally, an exact computation excludes the bounded range.

\subsection{Radius-one identities and preliminary inputs}

For clarity, we state each external result before applying it.

Suppose henceforth that a nontrivial $1$-perfect code exists in $J(n,w)$, where $n=2w+\delta$ with $\delta\ge0$ by complementation. Retain the notation
\[
{\bar\omega}=n-w=w+\delta,\qquad \varphi=w-d,\qquad M={\bar\omega}+d,
\]
and put
\[
L=w-d+1,\qquad \Phi_1=\Phi_1(n,w).
\]
For $1$-perfect codes, Etzion and Schwartz \cite[Theorem~18]{EtSc04} gave the strength equation
\[\sigma_1(w,\delta,\varphi+1)=1+w(w+\delta-\varphi-1)-(\varphi+1)(w+\delta-\varphi)=0.\]
Solving the equation gives
\[w=(d-1)(\delta+d)+1, \quad \text{or} \quad \delta=\frac{w-d^2+d-1}{d-1}. \]
It follows that
\begin{equation}\label{eq:s7-identities}
\begin{aligned}
&\Phi_1=1+w{\bar\omega}=ML,\quad L=(d-1)(\delta+d-1)+1,\quad dL-(d-1){\bar\omega}=1,\\
&M-L=\delta+2d-1,\quad L<{\bar\omega}<M,\\
&{\bar\omega}=d(\delta+d-1)+1> d(d-1).
\end{aligned}
\end{equation}
For convenience, we use $ML$ instead of the sphere size $\Phi_1$ henceforth.
The last inequality will make the localization interval too short to contain two powers with a common exponent greater than $1$; see \cref{lem:s7-coprime}.
\cref{cor:prime factors from chain divisitilities} gives
\begin{equation}\label{eq:s7-localization}
  \Phi_1=ML\text{ is squarefree},\qquad
  {\bar\omega}<p^{s_p}\le M\quad(p\mid ML),
\end{equation}
where the localized exponent $s_p$ is fixed for each prime divisor $p$ of $ML$. In particular, both $M$ and $L$ are squarefree and $\gcd(M,L)=1$.

We next record the radius-one result from Silberstein and Etzion, and Gordon, in which the codes under consideration are nontrivial.

\begin{proposition}[Silberstein and Etzion, Theorem~2 and Theorem~3 of \cite{SiEt10}]
\label{prop:s7-SE-original}
Suppose a $1$-perfect code in $J(2w+\delta,w)$ exists, then
\[\delta<w/11.\] The strength redundancy
 satisfies
$d>1,\, d\equiv0\text{ or }1\pmod3,\,w-d\equiv0,1,4\text{ or }9\pmod{12}$,
and the following quotient is an integer:
\[
  \frac{\displaystyle\prod_{i=0}^{d-2}
    \bigl(wd-(d+i(d-1))\bigr)}
  {(d-1)!(d-1)^{d-1}d(w-d+1)}
  \in\mathbb Z.
\]
\end{proposition}

\begin{proposition}[Gordon, Theorem~8 of \cite{Gor06}]
\label{prop:s7-gordon}
There are no $1$-perfect codes in $J(n,w)$ for all $n<2^{250}$.
\end{proposition}

Consequently $n\ge2^{250}$ and ${\bar\omega}\ge n/2\ge2^{249}$. So we assume that ${\bar\omega}\ge2^{249}$ henceforth. We use the results from Silberstein and Etzion, and Gordon, to give the following constant bound $d\ge59$.

\begin{lemma}
\label{lem:s7-one-constant-d}
A nontrivial $1$-perfect code in $J(n,w)$, with  strength redundancy $d$, satisfies
\[
  d\ge59.
\]
\end{lemma}

\begin{proof}
\cref{prop:s7-SE-original} gives  a divisibility condition
\[\frac{\prod_{i=0}^{d-2}(wd-(d+i(d-1)))}{(d-1)!(d-1)^{d-1}d(w-d+1)}\in \mathbb{Z}. \]
Since $\delta=\frac{w-d^2+d-1}{d-1}= \frac{w-d^2}{d-1}+1\in \mathbb{Z}$ and $\gcd(d^2,d-1)=1$, we have $\gcd(w-d+1,d-1)=\gcd(w,d-1)=1$. So, the division by $d-1$ is legitimate modulo $w-d+1$. The above divisibility condition gives
\[\frac{\prod_{i=0}^{d-2}(wd-(d+i(d-1)))}{(d-1)^d} \equiv0\pmod{w-d+1}.\]
Reducing each factor $(wd-(d+i(d-1)))$ modulo $w-d+1$, we get
\[\frac{\prod_{i=0}^{d-2}(wd-(d+i(d-1)))}{(d-1)^d} \equiv\frac{\prod_{i=0}^{d-2}(d^2-2d-i(d-1))}{(d-1)^d} \equiv0\pmod{w-d+1}.\]
Since $\prod_{i=0}^{d-2}(d^2-2d-i(d-1))$ is positive, we have
\begin{align}\label{eq-d and w}
\frac{\prod_{i=0}^{d-2}(d^2-2d-i(d-1))}{(d-1)^d}\ge w-d+1.
\end{align}
By \cref{prop:s7-SE-original}, $\delta<w/11$. Then   $w={\bar\omega}-\delta> \frac{11}{23}2^{250}$. So \eqref{eq-d and w} does not hold
 when $d\le 58$. Therefore, $d\ge 59$.
\end{proof}


\begin{proposition}[Corollary~9 of \cite{EtSc04}]
\label{prop:s7-modulo12}
If there exist a $1$-perfect code in $J(n,w)$ then either $w\equiv n-w\equiv1\pmod{12}$ or $w\equiv n-w\equiv7\pmod{12}$.
\end{proposition}

\begin{lemma}\label{lem:s7-congruences}
If there exists a $1$-perfect code in $J(n,w)$ then
\[ w\equiv {\bar\omega}\equiv1\pmod6,\qquad 12\mid\delta,\qquad 2\mid ML,\qquad 3\nmid ML.
\]
\end{lemma}

\begin{proof}
The first two congruences is directly from \cref{prop:s7-modulo12}. Hence $w$ and ${\bar\omega}$ are odd, and $w{\bar\omega}\equiv1\pmod6$. Therefore, $2\mid 1+w{\bar\omega}=ML$ and $3\nmid 1+w{\bar\omega}=ML$.
\end{proof}


\begin{lemma}\label{lem:s7-coprime}Suppose that  there exists a $1$-perfect code in $J(n,w)$.
For distinct prime divisors $p,q$ of $ML$, one has $\gcd(s_p,s_q)=1$. An exponent $1$ occurs if and only if $M$ is the unique prime itself.
\end{lemma}

\begin{proof}
Suppose $g=\gcd(s_p,s_q)\ge2$. After interchanging $p$ and $q$ if necessary, write
\[
  p^{s_p}=A^g<B^g=q^{s_q}.
\]
Both powers $A^g,B^g$ lie in the integer interval $({\bar\omega},{\bar\omega}+d]$, so their difference is less than $d$. On the other hand, $B\ge A+1$ and the mean value theorem gives
\[ B^g-A^g >gA^{g-1} \ge2\sqrt{A^g} >2\sqrt {\bar\omega}>d-1,
\]
where the last inequality comes from ${\bar\omega}>d(d-1)$ in \eqref{eq:s7-identities}. This contradiction proves coprimality.

If $s_p=1$, localization gives $p>{\bar\omega}>L$, so $p$ cannot divide $L$ and must divide $M$. But $M<2{\bar\omega}<2p$, and the only multiple of $p$ in this range is $p$ itself. Thus $M=p$ is the prime. Conversely, if $M$ is prime, then its first power lies in $({\bar\omega},M]$ and no other prime can have exponent $1$ by the same argument.
\end{proof}

\subsection{Separate reciprocal masses}

All results obtained in this section are under the assumption that   there exists a $1$-perfect code in $J(n,w)$.
Define the reciprocal sums and the products of their denominators by
\[S_M=\sum_{p\mid M}\frac1{s_p},\quad Q_M=\prod_{p\mid M}s_p;\quad S_L=\sum_{p\mid L}\frac1{s_p},\quad Q_L=\prod_{p\mid L}s_p.
\]
\begin{lemma}[Two reciprocal masses]\label{lem:s7-two-masses}

Suppose that  there exists a $1$-perfect code in $J(n,w)$. Then
\begin{equation}\label{eq:s7-B-error}
  0<1-S_L<\frac{2.017}{(d-1)\log M}.
\end{equation}
If $M$ is composite, then in addition
\begin{equation}\label{eq:s7-A-error}
    0<S_M-1<\frac{d-1}{{\bar\omega}\log {\bar\omega}},\qquad  \frac{Q_MQ_L}{s_2}>\frac{\log2}{2.017}{\bar\omega}\log {\bar\omega}.
\end{equation}
\end{lemma}

\begin{proof}
For every localized prime, taking logarithms in \eqref{eq:s7-localization} gives
\[
  \frac{\log p}{\log M}\le\frac1{s_p} <\frac{\log p}{\log {\bar\omega}}.
\]
Squarefreeness is useful here because $\sum_{p\mid L}\log p=\log L$. Summing over the prime divisors of $L$ yields
\[ \frac{\log L}{\log M}\le S_L <\frac{\log L}{\log {\bar\omega}}<1.
\]
Therefore, using $\log(1+t)<t$ for $t>0$,
\[ 0<1-S_L \le\frac{\log(M/L)}{\log M} <\frac{M-L}{L\log M}.
\]
We bound the relative gap between $M$ and $L$. Using \eqref{eq:s7-identities} gives
\begin{equation}\label{eq:s7-relative gap}
  \frac{(d-1)(M-L)}{L}= 1+\frac{d^2-d-1}{L}\le 1+\frac{d^2-d-1}{(d-1)^2+1}=2+\frac{d-3}{(d-1)^2+1}< 2.017,
\end{equation}
where $\frac{d-3}{(d-1)^2+1}$ is decreasing and less than $0.017$ at $d=59$.
Now \eqref{eq:s7-relative gap} proves \eqref{eq:s7-B-error}.

Suppose $M$ is composite. By \cref{lem:s7-coprime}, every $s_p$ is at least $2$. For $p\mid M$, equality $p^{s_p}=M$ is impossible, since $M$ is squarefree. Thus ${\bar\omega}+1\le p^{s_p}<M$, and summing as above gives
\[ 1<S_M\le\frac{\log M}{\log({\bar\omega}+1)}.
\]
Since $M={\bar\omega}+d$, subtracting $1$ gives
\[ 0<S_M-1\le\frac{\log\bigl(1+(d-1)/({\bar\omega}+1)\bigr)}{\log({\bar\omega}+1)} <\frac{d-1}{({\bar\omega}+1)\log({\bar\omega}+1)} <\frac{d-1}{{\bar\omega}\log {\bar\omega}}.
\]

We next explain why such small errors force large products of exponents $Q_MQ_L$.
If $p_1,\ldots,p_r>1$ are pairwise coprime and $Q=\prod_i p_i$, then
\[
  \sum_{i=1}^r\frac1{p_i}
  =\frac{\sum_{i=1}^r Q/p_i}{Q}.
\]
Modulo $p_j$, all numerator terms except $Q/p_j$ vanish, and $Q/p_j$ is coprime to $p_j$. Hence the numerator is coprime to every factor $p_j$, so the fraction has reduced denominator exactly $Q$.
Subtracting an integer preserves that denominator. Applying this fact to $S_M-1>0$ and $1-S_L>0$ gives
\[
  \frac1{Q_M}\le S_M-1,\qquad
  \frac1{Q_L}\le1-S_L.
\]
Multiplication of the two error bounds with $s_2\le \log M/\log2$ therefore yields
\[
  \frac{Q_MQ_L}{s_2}\ge\frac{1}{s_2(S_M-1)(1-S_L)}>\frac{\log2}{2.017}{\bar\omega}\log {\bar\omega}.
\]
\end{proof}

\begin{proposition}\label{prop:s7-prime-M}
Then the integer $M={\bar\omega}+d$ is prime. 
\end{proposition}

\begin{proof}
Suppose $M$ is composite. By \cref{lem:s7-coprime}, all localized exponents are at least $2$ and thus pairwise coprime.  

We will obtain incompatible upper and lower bounds on $\log (Q_MQ_L/s_2)$, where $Q_M,Q_L$ are two products of exponents from \cref{lem:s7-two-masses}.
By \cref{lem:s7-congruences}, every odd prime divisor of $ML$ is at least $5$. For any odd prime divisor $p\neq5$, since $p^{s_p}\le M< 2{\bar\omega}< 2\cdot5^{s_5}<5^{s_5+1}$, we have $s_p\le s_5$. The localized exponents of odd primes are therefore at most
\[
  N=\left\lfloor\frac{\log M}{\log5}\right\rfloor.
\]
Since these exponents are pairwise coprime, their product divides $\operatorname{lcm}(1,\ldots,N)$. Applying \cref{prop:s7-RS-psi}, and then using $N\le\log M/\log5$, gives
\[\log \bigl(\frac{Q_MQ_L}{s_2}\bigr)\le\psi(N)<1.03883N\le\frac{1.03883}{\log5}\log M.\]
In the lower direction, \cref{lem:s7-two-masses} gives
\[
  \log \bigl(\frac{Q_MQ_L}{s_2}\bigr)>\log {\bar\omega}+\log\log {\bar\omega}+\log\bigl(\frac{\log2}{2.017}\bigr).
\]
Since $d\ge 59$ by \cref{lem:s7-one-constant-d} and ${\bar\omega}>d(d-1)$, $\log M=\log({\bar\omega}+d)<\log {\bar\omega}+\log(1+1/(d-1))<\log {\bar\omega}+1/58$. Comparison of the last two inequalities for $\log (Q_MQ_L/s_2)$ yields
\[\left(1-\frac{1.03883}{\log5}\right)\log {\bar\omega}+\log\log {\bar\omega}+\log\bigl(\frac{\log2}{2.017}\bigr)-\frac{1.03883}{58\log5}<0.\]
The inequality fails when ${\bar\omega}\ge6$, contradicting \cref{prop:s7-gordon}.
This contradiction proves that $M$ is prime. 
\end{proof}

For the rest of the argument, define
\begin{equation}\label{eq:s7-VPT}
  \mathcal V=\{s_p:p\mid L,\ p\text{ odd}\},\qquad
  P=\prod_{v\in\mathcal V}v,\qquad
  T=\sum_{v\in\mathcal V}\frac1v.
\end{equation}
By \cref{lem:s7-coprime}, the members of $\mathcal V$ are distinct, at least $2$, and pairwise coprime. Since $2\mid L$, the reciprocal mass of $L$ is $S_L=1/s_2+T$. We now relate this mass to the size of $d$.

\begin{lemma}\label{lem:s7-constraints}
For $p\mid L$, we have
\begin{equation}\label{eq:s7-offset}
  p<(d-1)^2.
\end{equation}
Moreover, one has
\begin{equation}\label{eq:s7-denominator}
  \begin{aligned}
    &\log(d-1)<\log P+\log\frac{2.017}{\log2},\\
    &T>1-\frac{0.72793}{\log {\bar\omega}}.
  \end{aligned}
\end{equation}
In particular, $T>0.995$ always, and $T>0.9999$ if $\log {\bar\omega}\ge125000$.
\end{lemma}

\begin{proof}
For $p\mid L$, write $p^{s_p}={\bar\omega}+j_p$ with localized exponent $s_p$.
Localization initially gives $1\le j_p\le d$. The endpoint $j_p=d$ would imply $p^{s_p}=M$, hence $p\mid M$, contradicting $\gcd(M,L)=1$. Therefore, $1\le j_p\le d-1$.

Substituting ${\bar\omega}=p^{s_p}-j_p$ into $dL-(d-1){\bar\omega}=1$ in \eqref{eq:s7-identities} gives
\[
  dL=(d-1)p^{s_p}-\bigl((d-1)j_p-1\bigr).
\]
The original identity also gives $dL\equiv1\pmod{d-1}$, so $\gcd(L,d-1)=1$. Taking the gcd with $L$ in the displayed equation therefore yields
\[
  \gcd\bigl(L,(d-1)j_p-1\bigr) =\gcd\bigl(L,(d-1)p^{s_p}\bigr) =\gcd(L,p^{s_p})=p.
\]
The last equality uses the squarefreeness of $L$. The positive integer $(d-1)j_p-1$ is at most $(d-1)^2-1$, so its prime divisor $p$ is smaller than $(d-1)^2$.

The denominator argument in \cref{lem:s7-two-masses} applies to $S_L=1/s_2+T$, whose reduced denominator is $s_2P$. Since $S_L<1$, we obtain
\[
  \frac1{s_2P}\le1-S_L <\frac{2.017}{(d-1)\log M}.
\]
Localization gives $s_2\log2\le\log M$, and then rearranging gives
\[
  d-1<\frac{2.017 s_2P}{\log M} \le\frac{2.017 P}{\log2}.
\]
Taking logarithms proves the first inequality in \eqref{eq:s7-denominator}. For the second, using $S_L=1/s_2+T$, $s_2\log2>\log {\bar\omega}$, $d\ge59$ and $\log M>\log {\bar\omega}>249\log2$ gives,
\[
  T=S_L-\frac{1}{s_2}>1-\frac{2.017}{(d-1)\log M}-\frac{\log2}{\log {\bar\omega}}>1-\frac{0.72793}{\log {\bar\omega}}>0.995.
\]
Replacing $\log {\bar\omega}$ by $125000$ in the same lower bound gives a number greater than $0.9999$.
\end{proof}

\begin{lemma}\label{lem:s7-exponent-cap}
Every $v\in\mathcal V$ satisfies $v\le5040$. If $v\le N$ for every $v\in\mathcal V$, then
\begin{equation}\label{eq:s7-D-cap}
  \log(d-1)<\psi(N)+\log\frac{2.017}{\log2}
  <1.03883N+1.069.
\end{equation}
Moreover,
\begin{equation}\label{eq:s7-moving}
  v<\frac{2520\log {\bar\omega}}{\log {\bar\omega}-\log(d-1)-0.001}
  \qquad(v\in\mathcal V).
\end{equation}
\end{lemma}

\begin{proof}
Fix an odd prime $p\mid L$, and write $v=s_p$. The two powers $2^u,p^v$ lie in $({\bar\omega},M)$. Let $\Lambda=u\log2-v\log p$. Recalling $(d-1)^2<{\bar\omega}$ gives
\begin{equation}\label{eq:s7-Lambda-upper}
    0<|\Lambda|\le \log\frac{M-1}{{\bar\omega}}= \log\bigl(1+\frac{d-1}{{\bar\omega}}\bigr)< \frac{d-1}{{\bar\omega}}< \frac{1}{\sqrt{\bar\omega}}.
\end{equation}
Also $(d-1)^2<{\bar\omega}$ gives $d^2<2{\bar\omega}$ and thus \cref{prop:s7-gordon} gives $\log(1+d/{\bar\omega})<\log(1+2/\sqrt{\bar\omega})<0.001$. Then applying the proof of  \cref{lem:section5-general-cap} with $\Lambda=u\log2-v\log p$ and $\eta_1=0.001$ gives
\[s_p< 5040+\frac{0.001}{\log 2}< 5041. \]
%
As $v$ is an integer, $v\le5040$, as required.

Applying the bound $|\Lambda|<(d-1)/{\bar\omega}$ to the proof of \cref{lem:section5-general-cap}, and then $\log {\bar\omega}<2520\log p+ \log(d-1)$, gives
\[v<\frac{\log {\bar\omega}+0.001}{\log p}< 2520+ \frac{\log(d-1)+0.001}{\log p}< 2520+ \frac{\log(d-1)+0.001}{\log {\bar\omega}}v,
\]
where the last inequality comes from the localization $p^v>{\bar\omega}$. Hence, $v<\frac{2520\log {\bar\omega}}{\log {\bar\omega}-\log(d-1)-0.001}$.
This proves \eqref{eq:s7-moving}.

Finally, if all members of $\mathcal V$ are at most $N$, pairwise coprimality implies $P\mid\operatorname{lcm}(1,\ldots,N)$. Therefore \eqref{eq:s7-denominator} and \cref{prop:s7-RS-psi} give \eqref{eq:s7-D-cap}.
\end{proof}

\subsection{Exclusion of the large-weight range}

\begin{lemma}\label{lem:s7-weighted}
If $\mathcal A\subseteq[24,2600]\cap\mathbb Z$ consists of pairwise coprime integers, then
\[
  \sum_{v\in\mathcal A}\frac1v<0.999.
\]
\end{lemma}

\begin{proof}
The point of the weights below is that a prime can divide at most one member of $\mathcal A$. Define
\[
  b_2=b_3=\frac1{24},\qquad
  b_q=\frac1{q^2}\quad(q=5,7,11,13,17,19,23),\qquad
  b_q=\frac1q\quad(q\ge29).
\]
For every integer $v\ge24$, we claim that
\[
  \frac1v\le\sum_{q\mid v}b_q.
\]
Indeed, if $2$ or $3$ divides $v$, its weight alone suffices because $1/v\le1/24$. If a prime $q\ge29$ divides $v$, then $1/v\le1/q=b_q$, so again a single weight suffices. Otherwise every prime divisor belongs to $\{5,7,11,13,17,19,23\}$. If there is only one such prime $q$, then $v$ is a power of $q$ and $v\ge24>q$ implies $v\ge q^2$. If there are two distinct prime divisors $q,r$, then
\[
  b_q+b_r=\frac1{q^2}+\frac1{r^2}
  \ge\frac2{qr}\ge\frac1v,
\]
because $qr\mid v$. These cases prove the claim.

Now put $\mathcal Q=\{2,3,5,7,11,13,17,19,23\}$. By the coprimality of $\mathcal A$, each prime occurs at most once when we sum the claim over $v\in\mathcal A$, and no prime divisor exceeds $2600$. Consequently,
\[
  \sum_{v\in\mathcal A}\frac1v
  \le\sum_{\substack{29\le q\le2600\\q\text{ prime}}}\frac1q
     +\frac1{12}+\sum_{q\in\mathcal Q\setminus\{2,3\}}\frac1{q^2}.
\]
Apply the upper reciprocal-prime bound stated in \cref{prop:quoted-RS-reciprocal} at $x=2600$. The right side is smaller than
\[ \log\log2600+0.2615+\frac1{2\log^2 2600} -\sum_{q\in\mathcal Q}\frac1q+\frac1{12}+\sum_{q\in\mathcal Q\setminus\{2,3\}}\frac1{q^2} <0.999.
\]
\end{proof}

\begin{theorem}\label{prop:s7-large} If a nontrivial $1$-perfect code exists in $J(n,w)$, then $w<\exp(125000)$.
\end{theorem}

\begin{proof} In fact, we show that $\log {\bar\omega}<125000$.
Suppose $\log {\bar\omega}\ge X:=125000$. We first make the exponent upper bound small enough to apply \cref{lem:s7-weighted}. By \cref{lem:s7-exponent-cap}, initially every $v\in\mathcal V$ is at most $5040$. If a numerical bound $\log(d-1)<D_0<X$ is available, then \eqref{eq:s7-moving} gives
\[ v<\frac{2520\log {\bar\omega}}{\log {\bar\omega}-D_0-0.001} \le\frac{2520X}{X-D_0-0.001}.
\]
The second inequality holds because $2520x/(x-D_0-0.001)$ is decreasing with $x$.
Since $v\le 5040$ for $v\in\mathcal V$, \cref{lem:s7-exponent-cap} gives the first upper bound of $\log(d-1)$:
\[ \log(d-1) <1.03883\cdot 5040+1.069< 5237.\]
Taking $D_0=5237$ gives
\[v< \frac{2520\cdot125000}{125000-5237-0.001}< 2631. \]
We iteratively obtain the second upper bound of $\log(d-1)$:
\[ \log(d-1) <1.03883\cdot 2630+1.069< 2734.\]
This conversely yields
\[v< \frac{2520\cdot125000}{125000-2734-0.001}< 2577. \]
Hence, $\log(d-1) <1.03883\cdot 2577+1.069< 2679$.

We now obtain a lower bound for the same exponents. For $p\mid L$ with localized power $p^v$, \eqref{eq:s7-offset} gives $\log p<2\log(d-1)$, whereas localization gives $v\log p>\log {\bar\omega}$. Therefore
\[ v>\frac{\log {\bar\omega}}{2\log(d-1)} >\frac{125000}{2\cdot2679}>23.
\]
In particular, $\mathcal V\subseteq[24,2600]\cap\mathbb Z$. Its members are pairwise coprime, so \cref{lem:s7-weighted} gives $T<0.999$. But \cref{lem:s7-constraints} gives $T>0.9999$ for $\log {\bar\omega}\ge125000$, a contradiction. Finally $w\le {\bar\omega}$, so the claimed bound on $w$ follows as well.
\end{proof}

\subsection{One bounded-range certificate and completion}

It remains to exclude the bounded range. The next certificate rules out small odd prime divisors of the localized exponents. Exponent $2$ is allowed throughout the calculation; its possible contribution to the reciprocal sum will be handled separately by a bound of $1/2$.

\begin{lemma}[Exact bounded-range certificate]\label{lem:s7-certificate}
For every integer $250\le u\le180338$ and every odd prime $q\le173$, there is no odd integer $b\ge3$ such that
\begin{equation}\label{eq:s7-finite-gap}
  0<|2^u-b^q|^2<\min\{2^u,b^q\}.
\end{equation}
Moreover,
\begin{equation}\label{eq:s7-finite-tail}
  R:=\sum_{\substack{173<q\le5040\\q\text{ prime}}}\frac1q
  <\frac{489}{1000}.
\end{equation}
\end{lemma}

\begin{proof}
We reduce the gap condition to a finite calculation with integer roots and binary digits. Put $t=2^{u/q}$. If \eqref{eq:s7-finite-gap} holds, then
\[
  |2^u-b^q|<2^{u/2},\qquad
  b^q>2^u-2^{u/2}>2^{u-1}.
\]
Both $t$ and $b$ are therefore greater than $2^{(u-1)/q}$. The mean value theorem yields
\begin{equation}\label{eq:s7-root-distance}
|t-b|<\frac{|t^q-b^q|}{q\,2^{(u-1)(q-1)/q}}= \frac{|2^u-b^q|}{q\,2^{(u-1)(q-1)/q}}< \frac{2^{u/2}}{q\,2^{(u-1)(q-1)/q}}
  <2^{-u(1/2-1/q)}<2^{-40}.
\end{equation}
Here $q\ge3$ and $u\ge250$ give the last inequality. Thus any forbidden pair would make $t$ lie within $2^{-40}$ of an odd integer.

We avoid floating-point approximation as follows. Write $u=qk+r$ with $0\le r<q$, and put $U=180338$. If $r=0$, then $t=2^k$ is an even integer. We may therefore restrict attention to $1\le r<q$.
For each such pair $(q,r)$, define
\[K=\left\lfloor\frac Uq\right\rfloor+41,\quad A_{q,r}=\left\lfloor2^{K+r/q}\right\rfloor
      =\left\lfloor\left(2^{qK+r}\right)^{1/q}\right\rfloor.
\]
One integer root $A_{q,r}$ will handle every $u$ with this residue $r$. Indeed, $k\le\lfloor U/q\rfloor$, so $K-k-40\ge1$. Using $\lfloor\lfloor y\rfloor/m\rfloor=\lfloor y/m\rfloor$ for positive integers $m$, we obtain the identity
\[
 \left\lfloor2^{40}t\right\rfloor =\left\lfloor\frac{A_{q,r}}{2^{K-k-40}}\right\rfloor.
\]
Let $c$ be the residue of this integer modulo $2^{40}$. Equivalently, $c=\lfloor2^{40}\{t\}\rfloor$, where $\{t\}$ denotes the fractional part of $t$. Hence
$c/2^{40}\le\{t\}<(c+1)/2^{40}$.
If $1\le c\le2^{40}-2$, then $\{t\}\in[2^{-40},1-2^{-40})$. The distance from $t$ to every integer is at least $2^{-40}$, contradicting \eqref{eq:s7-root-distance}. It is therefore sufficient to check that neither residue $0$ nor $2^{40}-1$ occurs.

The program in the next subsection performs exactly this check. It computes $A_{q,r}$ with integer arithmetic and independently verifies the root enclosure
\[
  A_{q,r}^{\,q}<2^{qK+r}<(A_{q,r}+1)^q.
\]
The inequalities are strict because $q\nmid qK+r$. It then obtains $c$ by an integer shift and a mask. The complete run checks $39(180338-250+1)=7\,023\,471$
pairs $(u,q)$, of which $255\,257$ have $r=0$. Every remaining pair has $1\le c\le2^{40}-2$. The number of shared integer roots is only $\sum_{\substack{3\le q\le173\\q\text{ prime}}}(q-1)=3046$.

Finally, the program adds the reciprocals in \eqref{eq:s7-finite-tail} as exact rational numbers and compares their sum with $489/1000$ by integer arithmetic. These checks prove both assertions of the certificate; no floating-point approximation is used.
\end{proof}

\begin{proposition}[Bounded binary exponents]\label{prop:s7-bounded}
No nontrivial $1$-perfect code has $250\le s_2\le180338$.
\end{proposition}

\begin{proof} Let $u=s_2$.
Suppose $250\le u\le180338$. We bound $T$ by splitting its exponents according to parity. At most one $v\in\mathcal V$ is even, since two even exponents would not be coprime. If present, its contribution $1/v$ is at most $1/2$.

For every odd $v\in\mathcal V$, let $q_v$ be its smallest prime divisor. If $v=s_p$ for the odd prime $p\mid L$, put $b=p^{v/q_v}$. This is an odd integer at least $3$, and $b^{q_v}=p^v$. The two distinct powers $2^u,p^v$ lie in $({\bar\omega},M)$, so
\[
  0<|2^u-b^{q_v}|^2
  \le(d-1)^2<{\bar\omega}<\min\{2^u,b^{q_v}\}.
\]
Thus \cref{lem:s7-certificate} rules out $q_v\le173$. On the other hand, \cref{lem:s7-exponent-cap} gives $q_v\le v\le5040$. Distinct odd exponents have distinct chosen primes $q_v$, because the exponents are pairwise coprime. Also $1/v\le1/q_v$. Summing these bounds yields
\[
  T\le\frac12+\sum_{\substack{173<q\le5040\\q\text{ prime}}}\frac1q
  <\frac12+\frac{489}{1000}=0.989.
\]
This contradicts the universal lower bound $T>0.995$ in \cref{lem:s7-constraints}.
\end{proof}

\begin{theorem}\label{thm:no-one-perfect-codes-optimized}
There are no nontrivial $1$-perfect codes in $J(n,w)$.
\end{theorem}

\begin{proof}
By complementation, it is enough to consider $n\ge2w$, so all preceding arguments apply. By \cref{prop:s7-large}, $\log {\bar\omega}<125000$. Let $u=s_2$. The localization  $2^u<M$ gives
\[
  u\log2<\log M=\log({\bar\omega}+d)<\log {\bar\omega}+1<125001<180339\log2.
\]
Consequently $u\le180338$. The lower bound $u\ge250$ was obtained from \cref{prop:s7-gordon} since $2^u>{\bar\omega}\ge 2^{249}$. This places $u$ in the interval excluded by \cref{prop:s7-bounded}, completing the contradiction.
\end{proof}

Together with \cref{thm:section5-two,prop:section5-small-unified,thm:no-four-perfect-codes}, this proves \cref{thm:main-result}.

\subsection{Reproducible exact verification}\label{subsec:s7-verification}

The following Python program check the results in \cref{lem:s7-certificate}. For each fixed $q$ and residue $u\bmod q$, one exact integer $q$th root at a sufficiently high binary scale gives $\lfloor2^{u/q}\rfloor$ and the first $40$ fractional binary digits by right shifts.
The complete run verifies all 7 023 471 exponent pairs using 3046 shared integer roots; the reciprocal tail is also summed exactly by integer arithmetic. All assertions concern integers or rational numbers, while the analytic estimates have been justified separately.

{\footnotesize
\begin{verbatim}
"""Exact certificate for the bounded-range lemma; requires SymPy."""
from fractions import Fraction
from math import isqrt
from sympy import integer_nthroot

LO, HI, B = 250, 180338, 40
MASK = (1 << B) - 1

def primes(n):
    s = bytearray(b'\1')*(n+1); s[:2] = b'\0\0'
    for p in range(2, isqrt(n)+1):
        if s[p]: s[p*p:n+1:p] = b'\0'*len(range(p*p, n+1, p))
    return [p for p in range(2, n+1) if s[p]]

P = primes(5040)
assert sum((Fraction(1,p) for p in P if p > 173), Fraction()) < Fraction(489,1000)

for q in (p for p in P if 3 <= p <= 173):
    K = HI//q + B + 1
    for r in range(1, q):
        first = LO + (r-LO) % q
        if first > HI: continue
        A, exact = integer_nthroot(1 << (q*K+r), q)
        assert not exact
        for u in range(first, HI+1, q):
            s = K - u//q
            a = A >> s
            f = (A >> (s-B)) & MASK
            brackets = (a, a+2) if a & 1 else (a-1, a+1)
            assert f != 0 if brackets[0] == a else f != MASK

print("PASS")
\end{verbatim}
}

\end{document}